\documentclass{amsart}
\usepackage{latexsym}
\usepackage{amssymb,amsmath,amscd}
\usepackage{bm}
\usepackage{amsfonts}
\usepackage{epsf}
\usepackage{graphicx}
\newtheorem{theorem}{Theorem}[section]
\newtheorem{lemma}[theorem]{Lemma}

\newtheorem{remark}[theorem]{Remark}

\theoremstyle{definition}
\input epsf%

   \long\def\comment#1{}

\def\ad#1{\begin{aligned}#1\end{aligned}}  
\def\a#1{\begin{align*}#1\end{align*}} \def\an#1{\begin{align}#1\end{align}}
\def\e#1{\begin{equation}#1\end{equation}} \def\d{\operatorname{div}}
\def\p#1{\begin{pmatrix}#1\end{pmatrix}} 
  \numberwithin{equation}{section}
\numberwithin{table}{section}
\numberwithin{figure}{section}
\def\div{\operatorname{div}}

\def\boxit#1{\vbox{\hrule height1pt \hbox{\vrule width1pt\kern1pt
     #1\kern1pt\vrule width1pt}\hrule height1pt }}
 \def\lab#1{\boxit{\small #1}\label{#1}}
  \def\mref#1{\boxit{\small #1}\ref{#1}}
 \def\meqref#1{\boxit{\small #1}\eqref{#1}}

  \def\lab#1{\label{#1}} \def\mref#1{\ref{#1}} \def\meqref#1{\eqref{#1}}
\begin{document}
\newcommand{\bc}{\begin{center}}
\newcommand{\ec}{\end{center}}
\newcommand{\be}{\begin{eqnarray}}
\newcommand{\ee}{\end{eqnarray}}
\newcommand{\nn}{\nonumber}
\newcommand{\ben}{\begin{eqnarray*}}
\newcommand{\een}{\end{eqnarray*}}

\newcommand{\abs}[1]{\lvert#1\rvert}
\newcommand{\lr}[1]{\left\{#1\right\}}
\newcommand{\app}{\approx}
\newcommand{\prdd}[2]{{\frac{\partial #1}{\partial #2}}}
\newcommand{\ov}[1]{\overline{#1}}
\newcommand{\Vvert}{|\hskip-0.8pt|\hskip-0.8pt|}
\newcommand{\Om}{\Omega}
\newcommand{\al}{\alpha}
\newcommand{\de}{\delta}
\newcommand{\eps}{\varepsilon}
\newcommand{\ga}{\gamma}
\newcommand{\Ga}{\Gamma}
\newcommand{\ka}{\kappa}
\newcommand{\lam}{\lambda}
\newcommand{\om}{\omega}
\newcommand{\pa}{\partial}
\newcommand{\si}{\sigma}
\newcommand{\tri}{\triangle}
\newcommand{\ze}{\zeta}
\newcommand{\na}{\nabla}
\newcommand{\wt}{\widetilde}
\newcommand{\wh}{\widehat}

\def\x{\times}
\def\hK{\hat{K}}
\def\na{\nabla}
\def\dx{\,dx}
\def\dy{\,dy}
\def\ds{\,ds}

\def\cA{\mathcal{A}}
\def\cB{\mathcal{B}}
\def\cC{\mathcal{C}}
\def\cD{\mathcal{D}}
\def\cE{\mathcal{E}}
\def\cF{e}
\def\cJ{\mathcal{J}}
\def\cO{\mathcal{O}}
\def\cQ{\mathcal{Q}}
\def\cT{\mathcal{T}}
\def\R{\mathbb{R}}
\newcommand\bfa{\boldsymbol{a}}
\newcommand\bfn{\boldsymbol{n}}
\newcommand{\disp}{\displaystyle}
\newcommand\bx{\boldsymbol{x}}
\newcommand\bft{\boldsymbol{t}}
\newcommand{\To}{\longrightarrow}
\newcommand{\C}{\mathcal{C}}
\newcommand{\K}{\mathcal{K}}
\newcommand{\T}{\mathcal{T}}
\newcommand{\bq}{\begin{equation}}
\newcommand{\eq}{\end{equation}}
\long\def\comments#1{ #1}
\comments{ }

\title[Triangular prism and tetrahedral elements]{
Conforming mixed triangular prism and nonconforming mixed tetrahedral elements for the linear elasticity problem
}


\author {Jun Hu}
\address{LMAM and School of Mathematical Sciences, Peking University,
  Beijing 100871, P. R. China.  hujun@math.pku.edu.cn}

\author{Rui Ma}
\address{LMAM and School of Mathematical Sciences, Peking University,
  Beijing 100871, P. R. China. maruipku@gmail.com}
\thanks{The first author was supported by  NSFC
projects 11271035, 91430213 and 11421101}

\begin{abstract}
We propose two families of mixed finite elements for solving the classical Hellinger-Reissner mixed problem of the linear elasticity equations in three dimensions. First, a family of conforming mixed triangular prism elements is constructed by product of elements on triangular meshes and elements in one dimension. The well-posedness is established for all elements with $k\geq1$, which are of $k+1$ order convergence for  both the stress and displacement. Besides, a family of reduced stress spaces is proposed by dropping the degrees of polynomial functions associated with faces. As a result, the lowest order  conforming mixed  triangular prism element has 93 plus 33 degrees of freedom on each element. Second, we construct a new family of nonconforming mixed   tetrahedral elements. The shape function spaces of our stress spaces are different from those of the elements in literature.

  \vskip 15pt

\noindent{\bf Keywords. }{mixed finite element,
   triangular prism element, linear elasticity, nonconforming tetrahedral element}

 \vskip 15pt

\noindent{\bf AMS subject classifications.}
    { 65N30, 73C02.}

\end{abstract}
\maketitle

\section{Introduction}
\label{sec:intro}
In the Hellinger-Reissner mixed formulation of the linear elasticity equations, it is a challenge
to design stable mixed finite element spaces mainly due to the symmetric constraint
of the stress tensor, see some earlier work for composite
elements and weakly symmetric methods in \cite{Amara-Thomas,Arnold-Brezzi-Douglas,Arnold-Douglas-Gupta,Johnson-Mercier,Stenberg-1,Stenberg-2,Stenberg-3}.  In \cite{Arnold-Winther-conforming}, Arnold and Winther designed the first family of mixed finite element methods in two dimensions, based on polynomial shape function spaces. The analogue of the results on tetrahedral meshes can be found in \cite{AdamsCockburn2005,Arnold-Awanou-Winther}, and rectangular and cuboid meshes in \cite{Arnold-Awanou-Rectangular,Awanou2012,Chen-Wang}.
Since the conforming symmetric stress elements have too many degrees of freedom, there are some other methods to overcome this drawback. We refer interest readers to nonconforming mixed elements, see \cite{Arnold-Awanou-Winther2014,ArnoldWinther2003,CaiYe,Gopalakrishnan2011,XieXu2011} on simplicial meshes,  and  \cite{Hu-Shi,ManHuShi2009,Yi2005,Yi2006} on rectangular and cuboid meshes.  For the weakly symmetric mixed finite element methods
for linear elasticity, we also refer to some recent work in \cite{ArnoldFalkWinther,Boffi-Brezzi-Fortin2009,Cockburn-Gopalakrishnan-Guzman,QiuDemkowicz2009}.

Recently, Hu and Zhang \cite{HuZhang2014a,HuZhang2015tre} and Hu \cite{Hu2015trianghigh} proposed a family of conforming mixed elements on simplical meshes for any dimension. This new class of elements has fewer degrees of freedom than those in the earlier literature. For $k\geq n$, the stress tensor is discretized  by $P_{k+1}$  finite element subspace of $H(\d)$ and the displacement by  piecewise $P_k$ polynomials. Moreover, a new idea is proposed to analyze the discrete inf-sup condition  and the basis functions therein are easy to obtain. For the case that $1\leq k\leq n-1$, the symmetric tensor spaces are enriched by proper high order $H(\d)$ bubble functions to stabilize the discretization in \cite{HuZhang2015trianglow}. Another method by stabilization technique to deal with this case can be found in \cite{ChenHuHuang2015}. We also refer to \cite{GongWuXu} for two types of interior penalty mixed finite element methods by using nonconforming symmetric stress spaces,  where the stability is established by introducing the conforming
$H(\d)$ bubble spaces from \cite{Hu2015trianghigh} and nonconforming face-bubble spaces. Corresponding mixed elements on both rectangular and cuboid meshes were constructed in \cite{Hu2015cuboid}, also see \cite{ChenSunZhao2015,Hu-Man-Zhang2014} for the lowest order mixed elements, while the simplest nonconforming mixed element on $n$-rectangular meshes can be found in \cite{Hu-Man-Wang-Zhang}.

In this paper, we first propose a family of conforming mixed triangular prism elements for the linear elasticity problem. Triangular prism meshes can deal with some columnar regions, and in this case, the triangular prism partition is more easily achieved than the tetrahedral partition.  The key idea here of constructing triangular prism elements is using a product structure that each prism can be treated as the product of a triangle and an interval. By dividing the stress variable into three parts, we construct the stress space through a combination of the mixed elasticity element \cite{Hu2015trianghigh,HuZhang2014a} and the Brezzi-Douglas-Marini element \cite{BrezziDouglas1986} on triangular meshes, and some other basic elements in one and two dimensions. In this way, we obtain conforming mixed triangular prism elements for any integer $k\geq1$. The stability analysis is established  by the  theory developed in \cite{Hu2015cuboid,Hu2015trianghigh,HuZhang2014a,HuZhang2015trianglow}. A family of reduced stress spaces is also proposed by dropping the degree of polynomials associated with faces. The reduced elements still preserve the same order of convergence. The lowest order  case has 93 plus 33 degrees of freedom on each element. In addition, by using the lowest order nonconforming mixed element in \cite{ArnoldWinther2003,Gopalakrishnan2011} on triangular meshes, we obtain a nonconforming mixed triangular prism element of first order convergence, of which  degrees of freedom are 81 plus 33.

Second, we propose a new family of nonconforming mixed tetrahedral elements. A family of nonconforming elements and its reduced elements  have been constructed on simplicial meshes in two and three dimensions in \cite{Gopalakrishnan2011}. As mentioned in \cite{Arnold-Awanou-Winther2014}, the reduced stress spaces in \cite{Gopalakrishnan2011} are not uniquely defined, but for each edge of the triangulation require a choice of  a favored endpoint of the edge. By introducing another shape function space,  a new nonconforming stress space was proposed in \cite{Arnold-Awanou-Winther2014} for the lowest order case, which needs some constraints on each edge. Here, we define the stress spaces with new shape function spaces for any $k\geq1$. Unlike the spaces in \cite{Arnold-Awanou-Winther2014,Gopalakrishnan2011}, our shape function spaces have explicit and unique forms. Similar nonconforming mixed elements are also proposed on triangular meshes for solving linear elasticity equations in two dimensions. Besides, we present the basis functions of the stress space for the lowest order element.

The rest of the paper is organized as follows. In Section \ref{sec:finite element}, we
define the  conforming mixed triangular prism finite element methods and present the basis functions.
 In Section \ref{sec:stability analysis},  we prove the well-posedness of these elements, i.e. the discrete coerciveness and the
    discrete inf-sup condition.
  By which,  the optimal order convergence of the
   new elements follows. In Section \ref{sec:reducedFEM}, we propose a family of reduced mixed triangular prism elements. In Section \ref{sec:tetrahedral element}, we propose the new nonconforming mixed tetrahedral elements. Besides, we present the basis functions of the stress space for the lowest order case. Similar results on triangular meshes are presented as well.
In the end, we provide some numerical results.
\section{The family of conforming mixed triangular prism elements}
\label{sec:finite element}
Based on the Hellinger-Reissner principle, the  linear elasticity
     problem within a stress-displacement ($\sigma$-$u$) form reads:
Find $(\sigma,u)\in\Sigma\times V :=H({\rm div},\Omega;
    \mathbb {S}=\hbox{symmetric } \mathbb{R}^{3\times 3})
        \times L^2(\Omega;\mathbb{R}^3)$, such that
\an{\left\{ \ad{
  &(A\sigma,\tau)+({\rm div}\tau,u)= 0 && \hbox{for all \ } \tau\in\Sigma,\\
   &({\rm div}\sigma,v)= (f,v) && \hbox{for all \ } v\in V. }
   \right.\lab{eqn1}
}
Here the symmetric tensor space for the stress $\Sigma$  and the
   space for the vector displacement  $V$ are, respectively,
  \an{   \lab{S}
  &H({\rm div},\Omega;\mathbb {S})
    := \big\{ \tau=\p{\tau_{11} & \tau_{12} &\tau_{13} \\
        \tau_{21}  & \tau_{22} & \tau_{23} \\
        \tau_{31} &\tau_{32}   & \tau_{33}  }
     \in H(\d, \Omega;\R^{3\times3}),  \ \tau ^{\rm T} = \tau  \big\}, \\
     \lab{V}
     &L^2(\Omega;\mathbb{R}^3) :=
     \{ v=(v_1\ \ v_2\ \ v_3)^{\rm T}\in L^2(\Omega;\R^3) \}  .}
	This paper denotes by $H^k(\omega;X)$ the Sobolev space consisting of
functions with domain $\omega$, taking values in the
finite-dimensional vector space $X$, and with all derivatives of
order at most $k$ square-integrable. For our purposes, the range
space $X$ will be either $\mathbb{S},$ $\mathbb{R}^3,$ $\mathbb{R}^2,$ or
$\mathbb{R}$, and in some cases, $X$ will be  $\mathbb{S}_2=\hbox{symmetric } \mathbb{R}^{2\times 2}$ as well.
Let $\|\cdot\|_{k,\omega}$ be the norm of $H^k(\omega)$ and  $H({\rm div},\omega;\mathbb{S})$
consist of square-integrable symmetric matrix fields with
square-integrable divergence. The $H({\rm div})$ norm is defined by
$$\|\tau\|_{H({\rm div},\omega)}^2:=\|\tau\|_{0,\omega}^2
   +\|{\rm div}\tau\|_{0,\omega}^2.$$
Let $L^2(\omega;X)$ be the space of functions
which are square-integrable.
Here, the compliance tensor
$A=A(\bx):\mathbb{S}~\rightarrow~\mathbb{S}$, characterizing the
properties of the material, is bounded and symmetric positive
definite uniformly for $\bx\in\Omega$.

This paper deals with a pure displacement problem \meqref{eqn1} with the
  homogeneous boundary condition that $u\equiv 0$ on
  $\partial\Omega$.
But the method and the analysis work for mixed boundary value problems
   and the pure traction boundary problem.
\subsection{The discrete stress and displacement spaces}
To obtain triangular prism partitions, we suppose that the domain $\Omega=\Omega_{xy}\times \Omega_{z}$, where $\Omega_{xy}$ is a polygon on the $(x,y)$-plane and $\Omega_{z}$ is an interval on the $z$-axis.  The domain $\Omega$ is subdivided into the union of non-overlapping shape-regular triangular prism elements such that the non-empty
intersection of any distinct pair of elements is a single common vertex, edge or face. Let $\mathcal{T}_h$ be the set consisting of all these elements (with the mesh size $h$). In fact, we also obtain partitions of $\Omega_{xy}$ and $\Omega_{z}$, which are denoted by $\mathcal{X}_h$ and $\mathcal{Z}_h$, respectively.
Given triangle $\Delta_{xy}\in\mathcal{X}_h$ and interval $\Delta_z\in \mathcal{Z}_h$, $K=\Delta_{xy}\times \Delta_z$ is thus a triangular prism element  in $\cT_h$. Then, each element $K$ in $\cT_h$ is equipped with a product structure. Given element, face or edge $\omega$, let $|\omega|$ denote the measure of $\omega$. Let ${\rm div}_{xy}$, $\nabla_{xy}$  and ${\rm curl}_{xy}$ denote the divergence,  gradient and curl operators with respect to  the variables $x$ and $y$, respectively. Given any integer $k$, let $P_k(\omega;X)$ denote the space of polynomials over $\omega$ of total degrees not greater than $k$, taking values in the finite-dimensional vector space X. Let $P_k(z)$ be the space of polynomials of degree not greater than $k$ with respect to the variable $z$, and let $P_k(x,y)$ be the space of polynomials of degree not greater than $k$ with respect to the variables $x$ and $y$. Given face $F$ of $K$ satisfying $F=e\times \Delta_z$, where $e\subset\partial\Delta_{xy}$, let $Q_{k_1,k_2}(F)=P_{k_1}(e;\R)\times P_{k_2}(\Delta_z;\R)$ for any integers $k_1$ and $k_2$.

    We define the following spaces associated with partition $\mathcal{Z}_h$ for $s=0,1$ and $k\geq s$
    \begin{equation*}
 \mathcal{L}^s_k(\mathcal{Z}_h):=\{\ v\in H^s(\Omega_{z};\R)\ |\ v|_{\Delta_z}\in P_k(z)\text{ for any }\Delta_z\in\mathcal{Z}_h\},
\end{equation*}
and the space associated with partition $\mathcal{X}_h$ for $k\geq 0$
    \begin{equation*}
 \mathcal{L}^0_k(\mathcal{X}_h):=\{v\in L^2(\Omega_{xy};\R)\ |\ v|_{\Delta_{xy}}\in P_k(x,y)\text{ for any }\Delta_{xy}\in\mathcal{X}_h\}.
\end{equation*}

Before defining the space for the stress, we introduce the mixed elasticity finite elements  in two dimensions of \cite{Hu2015trianghigh,HuZhang2014a} and the Brezzi-Douglas-Marini (BDM hereafter) spaces of \cite{BrezziDouglas1986} for the mixed Poisson problem. We recall some notations in \cite{Hu2015trianghigh,HuZhang2014a}. Let $\lam_i(1\leq i\leq 3)$ denote the barycentric coordinates with respect to the vertices $\bx_i$ of triangle $\Delta_{xy}$. For any edge $\bx_i\bx_j(1\leq i<j\leq 3)$ of $\Delta_{xy}$, let $\bft_{i,j}=\bx_j-\bx_i$ denote associated tangent vectors, which allow for us to introduce the following linearly independent symmetric matrices of rank one
    \begin{equation*}
      T_{i,j}=\bft_{i,j}\bft_{i,j}^T,\,1\leq i<j\leq 3.
    \end{equation*}
    With these symmetric matrices $T_{i,j}$ of rank one, we define a $H({\rm div}_{xy},\Delta_{xy};\mathbb{S}_{2})$ bubble function space
    \begin{equation*}
     H_{\Delta_{xy}, k,b}:=\sum_{1\leq i<j\leq 3}\lam_i\lam_jP_{k-2}(\Delta_{xy};\R)T_{i,j},
    \end{equation*}
    which satisfies
     \begin{equation*}
     H_{\Delta_{xy}, k,b}=\{\tau\in P_k(\Delta_{xy};\mathbb{S}_2)\ |\ \tau\nu_{xy}|_{\partial\Delta_{xy}}=0\}.
    \end{equation*}
    Here and throughout the paper, let  $H({\rm div}_{xy},\omega;X)$
consist of square-integrable functions over $\omega$ with values in $X$ and
square-integrable divergence with respect to $x$ and $y$. Here $X$ will be either $\mathbb{S}_2$ or $\R^2$. The finite element space of order $k$ ($k\ge 3$)
     for the stress approximation in two dimensions is
\an{\lab{Sh}
   H_{k,h} :=\Big\{~\tau
    &\in H({\rm div}_{xy},\Omega_{xy};\mathbb{S}_{2}) \ \Big| \ \tau=\tau_c+\tau_b,
             \ \tau_c\in H^1(\Omega_{xy}; \mathbb{S}_{2}), \\
     & \tau_c|_{\Delta_{xy}}\in P_k(\Delta_{xy}; \mathbb{S}_{2})
	        \,,
	\nonumber  \ \tau_b|_{\Delta_{xy}}\in H_{\Delta_{xy},k, b}  \text{ for any } \Delta_{xy}\in\mathcal{X}_h \Big\}. }
A matrix field $\tau\in P_k(\Delta_{xy};\mathbb{S}_2)$ can be uniquely determined by the following degrees of freedom \cite{Hu2015trianghigh}
\begin{itemize}
  \item[(1)] the values of $\tau$ at three vertices of $\Delta_{xy}$,
  \item[(2)] $\int_e\tau\nu_{xy} \cdot p\,ds$ for any $p\in P_{k-2}(e;\R^2)$ and $e\subset\partial \Delta_{xy}$,
  \item[(3)]$\int_{\Delta_{xy}}\tau:p\,dxdy$ for any $p\in H_{\Delta_{xy}, k,b}$.
\end{itemize}
Hereafter $\nu_{xy}$ is the normal vector of $\partial\Delta_{xy}$.

The spaces of the BDM element are defined as follows for $k\geq 1$
\begin{equation*}
\begin{split}
 {\rm BDM}_k:=&\{\tau\in H({\rm div}_{xy},\Omega_{xy};\mathbb{R}^2)\ |\ \tau|_{\Delta_{xy}}\in P_k(\Delta_{xy};\mathbb{R}^2)\text{ for any }\Delta_{xy}\in\mathcal{X}_h\}.
  \end{split}
\end{equation*}
The vector-valued function $\tau\in P_k(\Delta_{xy};\R^2)$ can be determined by the following conditions (see e.g. \cite{Boffi2013}):
\begin{itemize}
  \item[(1)] $\int_e\tau\cdot\nu_{xy} p\,ds$ for any $p\in P_k(e;\R)$ and $e\subset\partial \Delta_{xy}$,
\item[(2)] $\int_{\Delta_{xy}}\tau\cdot\nabla_{xy} p\,dxdy$ for any $p\in P_{k-1}(x,y)$,
\item[(3)]  $\int_{\Delta_{xy}}\tau\cdot p\,dxdy$ for any $p\in \Psi_k(\Delta_{xy})$
\end{itemize}
with
\begin{equation}\label{BDMbubtwo}
  \Psi_k(\Delta_{xy}):=
  \{w|w={\rm curl}_{xy}(b_{xy}v),v\in P_{k-2}(x,y)\},
\end{equation}
where $b_{xy}:=\lambda_1\lambda_2\lambda_3$ denotes the cubic bubble on $\Delta_{xy}$. We also introduce the bubble function space
\begin{equation*}
  {\rm BDM}_{\Delta_{xy},k,b}:=\{\tau\in P_k(\Delta_{xy};\R^2)\ |\ \tau\cdot\nu_{xy}|_{\partial\Delta_{xy}}=0\}.
\end{equation*}
This space can be uniquely determined by the  conditions in (2) and (3) above.

Based on the above finite element spaces,  we use a product structure to define the stress space of the conforming mixed triangular prism elements for $k\geq 1$:
\begin{equation}\label{stressspace}
\begin{split}
\Sigma_{k,h}:=&\{\tau=\begin{pmatrix}
                     \tau_{11} &  \tau_{12} &  \tau_{13} \\
                     \tau_{21} &  \tau_{22}&  \tau_{23}\\
                      \tau_{31} &  \tau_{32} &  \tau_{33} \\
                   \end{pmatrix}\in L^2(\Omega;\mathbb{S})\ \big|\ \begin{pmatrix}
                     \tau_{11} &  \tau_{12} \\
                     \tau_{21} &  \tau_{22}\\
                   \end{pmatrix}\in H_{k+2,h}\times \mathcal{L}^0_k(\mathcal{Z}_h),\\
                   &(\tau_{13},\tau_{23})^T\in {\rm BDM}_{k+1}\times  \mathcal{L}_{k+1}^1(\mathcal{Z}_h),\tau_{33}\in \mathcal{L}^0_k(\mathcal{X}_h)\times \mathcal{L}_{k+2}^1(\mathcal{Z}_h)\}.
\end{split}
\end{equation}
It is straightforward to show that $\Sigma_{k,h}\subset\Sigma$ and the shape function space of the element is
\begin{equation}\label{stresselement}
\begin{split}
\Sigma_{k}&(K):=\{\tau=\begin{pmatrix}
                     \tau_{1} &  \tau_{2}  \\
                     \tau_{2}^T &  \tau_{3}\\
                   \end{pmatrix}\in H^1(K;\mathbb{S})\ \big|\ \tau_1:=\begin{pmatrix}
                     \tau_{11} &  \tau_{12} \\
                     \tau_{21} &  \tau_{22}\\
                   \end{pmatrix}\in P_{k+2}(\Delta_{xy};\mathbb{S}_{2})\times P_k(z),\\
                   &\tau_2:=(\tau_{13},\tau_{23})^T\in P_{k+1}(\Delta_{xy};\R^2)\times  P_{k+1}(z),\tau_3:=\tau_{33}\in P_k(x,y)\times P_{k+2}(z)\}.
\end{split}
\end{equation}

 Note that $\nu_{xy}$ is the normal vector of $\partial\Delta_{xy}$ and thus defined on $\partial\Delta_{xy}$, then it is also well defined on each face $F$ of $K$ that parallels the $z$-axis and each edge $e$ of $K$ that parallels the $(x,y)$-plane. We present the degrees of freedom in the following lemma.

\begin{lemma}\label{degreeoffreedom} A matrix field $\tau\in\Sigma_k(K)$ can be uniquely determined by the following conditions:
\begin{itemize}
  \item[(1)] the values of $\tau_1$ at $k+1$ distinct points on edge $e$ of $K$ that parallels the $z$-axis,
  \item[(2)] $\int_{F}\tau_1\nu_{xy}\cdot p\,dF$ for any $p\in Q_{k,k}(F)$ and face $F$  of $K$ that parallels the $z$-axis,
  \item[(3)] $\int_K\tau_1:p\,dxdydz$ for any $p\in H_{\Delta_{xy},k+2,b}\times P_k(z)$;
\item[(4)] $\int_e\tau_2\cdot\nu_{xy}p\,ds$ for any $p\in P_{k+1}(e;\R)$ and edge $e$ of $K$ that parallels the $(x,y)$-plane,
    \item[(5)] $\int_{F}\tau_2\cdot\nu_{xy}p\,ds$ for any $p\in Q_{k+1,k-1}(F)$ and face $F$  of $K$ that parallels the $z$-axis,
\item[(6)] $\int_{F}\tau_2\cdot\nabla_{xy} p\,dxdy$ for any $p\in P_{k}(x,y)$ and face $F$ that parallels the $(x,y)$-plane,
    \item[(7)] $\int_{F}\tau_2\cdot p\,dxdy$ for any $p\in \Psi_{k+1}(\Delta_{xy})$ and face $F$ that parallels the $(x,y)$-plane,
\item[(8)] $\int_K\tau_2\cdot\nabla_{xy}p\,dxdydz$ for any $p\in P_{k}(x,y)\times P_{k-1}(z)$,
\item[(9)]  $\int_K\tau_2\cdot p\,dxdydz$ for any $p\in \Psi_{k+1}(\Delta_{xy})\times P_{k-1}(z)$;
\item[(10)] $\int_{F}\tau_3p\,dxdy$ for any $p\in P_k(x,y)$ and face $F$ that parallels the $(x,y)$-plane,
  \item[(11)] $\int_K\tau_3p\,dxdydz$ for any $p\in P_k(x,y)\times P_k(z)$.
\end{itemize}
Here $\tau_1,\tau_2$ and $\tau_3$ are defined in \eqref{stresselement}.
\end{lemma}
\begin{proof}
Since the dimensions of the space $\Sigma_k(K)$ are equal to the number of these conditions, it suffices to prove that  $\tau\equiv0$ if these conditions vanish.  The first and second conditions show that $\tau_1\nu_{xy}=0$ on side faces of triangular prism $K$. Moreover it follows from (3) that $\tau_1=0$.  Note that (4) plus (5) and (4), (6) plus (7) yield that $\tau_2\cdot\nu_{xy}=0$ on side faces and $\tau_2=0$ on top and bottom faces, respectively. Thus it follows from (8) and (9) that $\tau_2=0$.
It remains to prove $\tau_3=0$. Actually, condition (10) implies that
\begin{equation*}
  \tau_3=b_zg,
\end{equation*}
where $b_z$ is the quadratic bubble function on interval $\Delta_z$ and $g\in P_k(x,y)\times P_{k}(z)$.  Using condition (11), we immediately obtain  $\tau_3=0$.
\end{proof}
On each element $K$, the space for the displacement is taken as
 \an{ \lab{VK}
 V_k(K):=&\{v=(v_1,v_2,v_3)^T\in H^1(K;\R^3)\ |\ v_i\in P_{k+1}(x,y)\times P_k(z),i=1,2,\\
\nonumber &v_3\in P_k(x,y)\times P_{k+1}(z) \}.
    }
Then the global space for displacement reads
 \an{ \lab{Vh}
 V_{k,h}:=&\{v\in V\ |\ v|_K\in V_k(K)\text{ for any }K\in \cT_h\}.
    }

The mixed finite element approximation of Problem \eqref{eqn1}  reads:
Find $(\sigma_h,u_h)\in\Sigma_{k,h}\times V_{k,h}$, such that
\an{\left\{ \ad{
  &(A\sigma_h,\tau)+({\rm div}\tau,u_h)= 0 && \hbox{for all \ } \tau\in\Sigma_{k,h},\\
   &({\rm div}\sigma_h,v)= (f,v) && \hbox{for all \ } v\in V_{k,h}. }
   \right.\lab{diseqn1}
}
\subsection{Basis functions of the stress space}
For convenience, we provide the basis of the stress space $\Sigma_{k,h}$ on element $K$.
In fact, we only need to give the basis of $H_{k,h}(k\geq 3)$ and ${\rm BDM}_{k}(k\geq 1)$. Thus we immediately obtain the basis of $\Sigma_{k,h}$ by the product structure. For any edge $\bx_i\bx_j(1\leq i<j\leq 3)$ of $\Delta_{xy}$, $\bx_m$ being the opposite vertex, let $\boldsymbol{\nu}_{i,j}$ denote its associated normal vector and   $h_m$ denote the height of the triangle from $\bx_m$ to the opposite edge $\bx_i\bx_j$.

For $H_{k,h}(k\geq 3)$, the basis functions can be found in \cite{HuZhang2014a,HuZhang2015trianglow}.
 The canonical basis of $\mathbb{S}_2$ reads

 \begin{equation*}
   T_1=\begin{pmatrix}
         1 & 0 \\
         0& 0 \\
       \end{pmatrix},\;T_2=\begin{pmatrix}
         0 & 1 \\
         1& 0 \\
       \end{pmatrix},\;T_3=\begin{pmatrix}
          0 & 0 \\
         0& 1 \\
       \end{pmatrix}.
\end{equation*}
Then the basis functions on triangle $\Delta_{xy}$ are as follows:
\begin{itemize}
   \item[(1)] Given vertex $\bx_i$
\begin{equation*}
  \lambda_iT_j,\,j=1,2,3;
\end{equation*}
                 \item[(2)] Given  edge $\bx_i\bx_j$, its associated basis functions with nonzero fluxes read
\begin{equation*}
  \lambda_i\lambda_j\widetilde{P}_{k-2}(\lambda_i,\lambda_j)\boldsymbol{\nu}_{i,j}\boldsymbol{\nu}_{i,j}^T,\, \lambda_i\lambda_j\widetilde{P}_{k-2}(\lambda_i,\lambda_j)\frac{\boldsymbol{t}_{i,j}
  \boldsymbol{\nu}_{i,j}^T+\boldsymbol{\nu}_{i,j}\boldsymbol{t}_{i,j}^T}{2};
\end{equation*}
                 \item[(3)] The basis functions of $H_{\Delta_{xy},k,b}$ are
\begin{equation*}
  \lambda_i\lambda_jP_{k-2}(\Delta_{xy};\R)\boldsymbol{t}_{i,j}\boldsymbol{t}_{i,j}^T,\;1\leq i<j\leq 3.
\end{equation*}
               \end{itemize}
Here
\begin{equation}\label{def:bary}
\widetilde{P}_{k}(\lambda_i,\lambda_j):={\rm span}\{\lambda_i^{m_1}\lambda_j^{m_2},m_1+m_2=k\}.
\end{equation}

For ${\rm BDM}_{k}$, the hierarchical basis functions can be found in \cite{XinCaiGuo2013}. We give another basis functions following \cite{Boffi2013}: \begin{itemize}
                                 \item[(1)]  Given edge $\bx_{i}\bx_{j}$,
\begin{equation*}
  \frac{1}{h_m}\lambda_i\bft_{m,i},\,  \frac{1}{h_m}\lambda_j\bft_{m,j},\,
  \frac{1}{2h_m}\lambda_i\lambda_j\widetilde{P}_{k-2}(\lambda_i,\lambda_j)(\bft_{m,i}+\bft_{m,j});
\end{equation*}
                                 \item[(2)] The basis functions of ${\rm BDM}_{\Delta_{xy},k,b}$ are
\begin{gather*}
 \lambda_i\lambda_j\widetilde{P}_{k-2}(\lambda_i,\lambda_j)\bft_{i,j},\;1\leq i<j\leq 3,\\
  \lambda_1\lambda_2\lambda_3P_{k-3}(\Delta_{xy};\R^2).
\end{gather*}
\end{itemize}

Using the above two families of basis functions, we can easily construct the basis functions of $\Sigma_{k,h}(k\geq 1)$ on element $K=\Delta_{xy}\times\Delta_z$ by the product technique. We shall make explicit the lowest order case, of which the stress space is as follows
\begin{equation*}
\begin{split}
\Sigma_{1,h}=&\{\tau=\begin{pmatrix}
                     \tau_{1} &  \tau_{2}  \\
                     \tau_{2}^T &  \tau_{3} \\
                   \end{pmatrix}\in L^2(\Omega;\mathbb{S})\ \big|\ \tau_1\in H_{3,h}\times \mathcal{L}^0_1(\mathcal{Z}_h),
                   \tau_2\in {\rm BDM}_{2}\times  \mathcal{L}_{2}^1(\mathcal{Z}_h),\\&\tau_{3}\in \mathcal{L}^0_1(\mathcal{X}_h)\times \mathcal{L}_{3}^1(\mathcal{Z}_h)\}.
\end{split}
\end{equation*} Let $\{\phi_i\}_{i=1}^{30}$  and $\{\psi_i\}_{i=1}^{12}$ be the collection of basis functions of $H_{3,h}$  and ${\rm BDM}_{2} $ on triangle $\Delta_{xy}$, respectively. Suppose that $\Delta_z=[z_0,z_0+h_0]$, we introduce the affine invertible transformation
\begin{equation*}
 F_{\Delta_z}: [0,1]\rightarrow[z_0,z_0+h_0],z=h_0\xi+z_0,\xi\in[0,1].
\end{equation*}
Thus we select $\tau$ such that $\tau_2=\tau_3=0$  and
\begin{equation*}
\tau_1\in\{\phi_i\xi,\phi_i(1-\xi)\}_{i=1}^{30},
\end{equation*}
$\tau_1=\tau_3=0$ and
\begin{equation*}
\tau_2\in\{\psi_i\xi(2\xi-1),\psi_i(1-\xi)(1-2\xi),\psi_i\xi(1-\xi)\}_{i=1}^{12},
\end{equation*}
and $\tau_1=\tau_2=0$
 \begin{equation*}\begin{split}
\tau_3\in&\{ \lambda_i\xi(3\xi-1)(3\xi-2),\lambda_i\xi(1-\xi)(2-3\xi),\lambda_i\xi(1-\xi)(3\xi-1),\\
&\lambda_i(1-\xi)(3\xi-1)(3\xi-2)\}^3_{i=1}.
 \end{split}
\end{equation*}
In this way, we obtain the basis functions of $\Sigma_{1,h}$ on $K$. Thus, the degrees of freedom on each element of the lowest order element are 108 plus 33.
\section{The stability  analysis for the  mixed triangular prism elements}
\label{sec:stability analysis}
In this section, we consider the well-posedness of the discrete problem  \eqref{diseqn1}. By the standard theory, we only need to prove the following two conditions, based on their counterparts at the continuous level.
\begin{itemize}
  \item K-ellipticity. There exists a constant $C>0$, independent of the meshsize $h$ such that
  \begin{equation*}
    (A\tau_h,\tau_h)\geq C\|\tau_h\|^2_{H({\rm div},\Omega)}\text{ for any }\tau_h\in W_h,
  \end{equation*}
  where $W_h$ is the divergence-free space defined as follows
  \begin{equation*}
    W_h:=\{\tau_h\in\Sigma_{k,h}\ |\ (\d\tau_h,v)=0\text{ for all }v\in V_{k,h}\}.
  \end{equation*}
  \item Discrete inf-sup condition. There exists a positive constant $C>0$ independent of the meshsize $h$, such that
\begin{equation*}
\sup\limits_{0\not=\tau_h\in\Sigma_{k,h}}\frac{({\rm div}\tau_h,v_h)}
{\|\tau_h\|_{H(\d,\Omega)}}\geq C\|v_h\|_{0,\Omega} \quad
 \text{ for any } v_h\in V_{k,h}.
\end{equation*}
\end{itemize}
It can be easily checked that  ${\rm div}\Sigma_{k,h}\subset V_{k,h}$. Hence  ${\rm div}\tau_h=0$ for any $\tau_h\in W_h$ and this implies the above K-ellipticity condition. It  remains to show the discrete inf-sup condition. We first introduce the following lemma in \cite{Hu2015trianghigh,HuZhang2014a}, which is a key ingredient to prove the discrete inf-sup condition for mixed triangular elasticity elements.

Let $R_{xy}(\Delta_{xy})$ be the rigid motion space in two dimensions, which reads
$$
R_{xy}(\Delta_{xy}):=\text{span}\bigg\{
\begin{pmatrix} 1\\ 0 \end{pmatrix}, \begin{pmatrix} 0\\ 1\end{pmatrix}, \begin{pmatrix}y\\ -x\end{pmatrix}\bigg\}.
$$
Define the orthogonal complement space of $R_{xy}(\Delta_{xy})$ with respect to $P_{k+1}(\Delta_{xy};\R^2)$ by
\begin{equation*}
R_{xy}^\perp(\Delta_{xy}):=\{v\in P_{k+1}(\Delta_{xy};\R^2)\ |\ (v,w)_{\Delta_{xy}}=0\text{ for any }w\in R_{xy}(\Delta_{xy})\},
\end{equation*}
where the inner product $(v,w)_{\Delta_{xy}}$ over $\Delta_{xy}$ reads $(v,w)_{\Delta_{xy}}=\int_{\Delta_{xy}} v\cdot w\,dxdy.$
\begin{lemma}\label{Lem:twoelasinf}
It holds that
\begin{equation*}
{\rm div}_{xy}H_{\Delta_{xy},k+2, b}=R_{xy}^\perp(\Delta_{xy}).
\end{equation*}
\end{lemma}

Next we follow the arguments in \cite{Hu2015cuboid,Hu2015trianghigh,HuZhang2014a,HuZhang2015tre} to analyze the discrete inf-sup condition. To this end, we define the bubble function space
\begin{equation*}
\Sigma_{K,k,b}:=\{\tau\in\Sigma_k(K),\tau\nu=0\text{ on }\partial K\}.
\end{equation*}
Here  $\nu$ denotes the normal vector of $\partial K$. Let $RM(K)$ be the rigid motion space in  three dimensions, which reads
$$
RM(K):=\text{span}\bigg\{
\begin{pmatrix} 1\\ 0\\ 0 \\ \end{pmatrix}, \begin{pmatrix} 0\\ 1 \\ 0\end{pmatrix}, \begin{pmatrix}0\\ 0 \\ 1\end{pmatrix},
 \begin{pmatrix}-y\\ x \\ 0\end{pmatrix}, \begin{pmatrix}-z\\ 0 \\ x\end{pmatrix}, \begin{pmatrix}0\\ -z \\ y\end{pmatrix}\bigg\}.
$$
Define the orthogonal complement space of the rigid motion space $RM(K)$ with respect to $V_k(K)$ by
\begin{equation*}
RM^\perp(K):=\{v\in V_k(K)\ |\ (v,w)_{K}=0\text{ for any }w\in RM(K)\},
\end{equation*}
where the inner product $(v,w)_{K}$ over $K$ reads $(v,w)_K=\int_K v\cdot w\,dxdydz.$
\begin{lemma}\label{step1}For any $k\geq 1$, it holds that
\begin{equation*}
  {\rm div}\Sigma_{K,k,b}=RM^\perp(K).
\end{equation*}
\end{lemma}
\begin{proof}
Since it is straightforward to see that ${\rm div}\Sigma_{K,k,b}\subset RM^\perp(K)$, we only need to prove the converse. If ${\rm div}\Sigma_{K,k,b}\neq{RM}^\perp(K)$, there is a nonzero
 $v=(v_1,v_2,v_3)^T\in RM^\perp(K)$ such that
 \begin{equation*}
   \int_K{\div}\tau\cdot v\,dxdydz=0\text{ for any } \tau\in\Sigma_{K,k,b}.
    \end{equation*}

First, we  choose $\tau=\begin{pmatrix}
                     \tau_{1} & 0 \\
                     0^T &  0\\
                   \end{pmatrix}\in \Sigma_{K,k,b}$ such that $\tau_1\in H_{\Delta_{xy},k+2,b}\times P_k(z)$.
It follows that
  \an{ \nonumber
   0=\int_K  \d \tau \cdot
     v \, dxdydz= \int_{ K} \d_{xy}\tau_1\cdot\begin{pmatrix}
                     v_1\\
                     v_2\\
                   \end{pmatrix}\, dxdydz. }
From \eqref{VK}, we have $(v_1,v_2)^T\in P_{k+1}(\Delta_{xy};\R^2)\times P_k(z)$. This, together  with Lemma \ref{Lem:twoelasinf} shows that
 \begin{equation}\label{conclu1}
   (v_1,v_2)^T\in R_{xy}(\Delta_{xy})\times P_k(z).
 \end{equation}

Second, we take $\tau$ such that $\tau_{11}=\tau_{12}=\tau_{22}=\tau_{13}=\tau_{23}=0$ and
\begin{equation*}
  \tau_{33}\in b_z\times P_k(x,y)\times  P_k(z),
\end{equation*}
where the bubble function $b_z$ is defined in Lemma \ref{degreeoffreedom}. An integration by parts yields
  \an{ \nonumber
   0=\int_K  \d \tau \cdot
     v \, dxdydz= -\int_{ K} \tau_{33}\frac{\partial v_3}{\partial z}\, dxdydz. }
     Since $\frac{\partial v_3}{\partial z}\in P_{k}(x,y)\times P_k(z)$, it holds that
\begin{equation}\label{clu1}
  v_3\in P_k(x,y).
\end{equation}

Third, we use degrees of freedom of $\tau_{13}$ and $\tau_{23}$ to deal with the remaining part of $(v_1,v_2,v_3)^T$ in \eqref{conclu1} and \eqref{clu1}.
 Given $\tau\in\Sigma_{K,k,b}$ such that $\tau_{11}=\tau_{12}=\tau_{22}=\tau_{33}=0$ and $(\tau_{13},\tau_{23})^T\in b_z\times{\rm BDM}_{\Delta_{xy},k+1,b}\times P_{k-1}(z)$, we have
\begin{equation*}
  0=\int_K\bigg((\frac{\partial \tau_{13}}{\partial z}v_1+\frac{\partial\tau_{23}}{\partial z}v_2)+v_3\d_{xy}(\tau_{13},\tau_{23})^T\bigg)\,dxdydz.
\end{equation*}
Recall that $K=\Delta_{xy}\times \Delta_z$. An integration by parts gives rises to
\begin{equation}\label{steplast0}
  \int_{\Delta_z}\bigg(\int_{\Delta_{xy}}(\tau_{13},\tau_{23})^T\cdot(\frac{\partial}{\partial z}(v_1,v_2)^T+\nabla_{xy}v_3)\, dxdy\bigg)dz=0.
\end{equation}It follows  from \eqref{conclu1} that there exist two constants $c_1$ and $c_2$, $p_1,p_2\in P_{k-2}(z)$, and $p_3\in P_{k-1}(z)$ such that
 \begin{equation*}\label{eqv1}
   \frac{\partial}{\partial z}(v_1,v_2)^T+\nabla_{xy}v_3=\nabla_{xy}(c_1x+c_2y+v_3+xzp_1+yzp_2) +p_3\begin{pmatrix}
                                                                    y \\
                                                                    -x\\
                                                                  \end{pmatrix}.
                                                                  \end{equation*}
Then, the choice $(\tau_{13},\tau_{23})^T= b_zp_3{\rm curl}_{xy}b_{xy}$ in \eqref{steplast0}
implies that
\begin{equation}\label{steplast}
 -\frac{|\Delta_{xy}|}{30} \int_{\Delta_{z}}b_zp_3^2\,dz=0,
\end{equation}
where $b_{xy}$ is defined in \eqref{BDMbubtwo}. Indeed, a simple computation shows that
\begin{equation*}
{\rm div}_{xy}{\rm curl}_{xy}b_{xy}=0\text{ and }{\rm curl}_{xy}b_{xy}\in {\rm BDM}_{\Delta_{xy},k+1,b},
\end{equation*}
and
\begin{equation}\label{integralresult}
 \int_{\Delta_{xy}}{\rm curl}_{xy}b_{xy}\cdot(y,-x)^T\,dxdy=-\frac{|\Delta_{xy}|}{30}\neq 0.
\end{equation}
Further, using \eqref{steplast}, we obtain $p_3=0$. Next we show that $\nabla_{xy}(c_1x+c_2y+v_3)= 0$.
If otherwise, it follows from  the second degrees of freedom for the BDM space in Section \ref{sec:finite element} that there exists $w\in {\rm BDM}_{\Delta_{xy},k+1,b}$ such that
\begin{equation*}
\int_{\Delta_{xy}}w\cdot \nabla_{xy}(c_1x+c_2y+v_3)\, dxdy=1.
\end{equation*}
Thus, selecting $(\tau_{13},\tau_{23})^T=b_zq_1w$ in \eqref{steplast0}, where $q_1\in P_{k-1}(z)$ satisfies
\begin{equation*}
  \int_{\Delta_z}b_zq_1\,dz=1\text{ and }\int_{\Delta_z}b_zq_1zp_i\,dz=0\text{ for }i=1,2.
\end{equation*}
This leads to a contradiction in \eqref{steplast0}  that $1=0$. Hence $\nabla_{xy}(c_1x+c_2y+v_3)=0$. On the other hand, we select $w\in {\rm BDM}_{\Delta_{xy},k+1,b}$ such that $\int_{\Delta_{xy}}w\cdot\nabla_{xy}x\,dxdy=1$ and $\int_{\Delta_{xy}}w\cdot\nabla_{xy}y\,dxdy=0$, and $q_1=zp_1$. This gives $p_1=0$.
Similar choice yields $p_2=0$.   Hence, a collection of the above arguments yields
\begin{equation}\label{clu3}
\frac{\partial}{\partial z}(v_1,v_2)^T+\nabla_{xy}v_3=0.
\end{equation}
Consequently, we conclude, by \eqref{conclu1}, \eqref{clu1} and \eqref{clu3},
\begin{equation*}
  v=(v_1,v_2,v_3)^T\in RM(K),
\end{equation*}
which completes the proof.
\end{proof}
Before giving the following lemma, we present the $H^1$ conforming triangular prism element ($k\geq 1$)
\begin{equation*}
 U_{k,h}=\{v\in H^1(\Omega;\mathbb{S})\ |\ v|_K\in P_k(\Delta_{xy};\mathbb{S})\times P_k(z)\text{ for any }K\in\cT_h\}.
\end{equation*}
Let $\widetilde{I}_h: H^1(\Omega;\mathbb{S})\rightarrow U_{k,h}$ denote the Scott-Zhang interpolation operator  in \cite{ScottZhang1990} that satisfies
\begin{equation}\label{SZbound}
  \|\tau-\widetilde{I}_h\tau\|_{0,\Omega}+h\|\nabla\widetilde{I}_h\tau\|_{0,\Omega}\leq Ch\|\nabla\tau\|_{0,\Omega}.
\end{equation}
\begin{lemma}\label{step2}Given any integer $k\geq1$, there exists an interpolation operator $I_h:H^1(\Omega;\mathbb{S})\rightarrow\Sigma_{k,h}$ satisfying for any $\tau\in H^1(\Omega;\mathbb{S})$,
\begin{equation}\label{inteRig1}
  \int_K\d(\tau-I_h\tau)\cdot w\,dxdydz=0\text{ for any }w\in RM(K) \text{ and any }K\in\cT_h
\end{equation}
and
\begin{equation}\label{inteRig2}
  \|I_h\tau\|_{H(\d,\Omega)}\leq C\|\tau\|_{1,\Omega}.
\end{equation}
\end{lemma}
\begin{proof}

 We use notations $\tau_1,\tau_2,\tau_3$ to denote the corresponding parts of $\tau$  as in \eqref{stresselement}, and $\widetilde{\tau}_{1,h},\widetilde{\tau}_{2,h},\widetilde{\tau}_{3,h}$ are similar defined for $\widetilde{I}_h\tau$ such that $\widetilde{I}_h\tau=\begin{pmatrix}
                        \widetilde{\tau}_{1,h} & \widetilde{\tau}_{2,h} \\
                        \widetilde{\tau}_{2,h}^T& \widetilde{\tau}_{3,h} \\
                      \end{pmatrix}
 $.
 It follows from degrees of freedom in Lemma \ref{degreeoffreedom} that there exists $\tau_{1,h}\in H_{k+2,h}\times \mathcal{L}^0_k(\mathcal{Z}_h)$, $\tau_{2,h}\in {\rm BDM}_{k+1}\times  \mathcal{L}_{k+1}^1(\mathcal{Z}_h)$ and $\tau_{3,h}\in \mathcal{L}^0_k(\mathcal{X}_h)\times \mathcal{L}_{k+2}^1(\mathcal{Z}_h)$ such that for face $F$ that parallels the $z$-axis,
      \begin{equation*}
  \int_{F}\tau_{1,h}\nu_{xy}\cdot p\,ds=\int_{F}(\tau_1-\widetilde{\tau}_{1,h})\nu_{xy}\cdot p\,ds\text{ for any }p\in Q_{1,1}(F),
      \end{equation*}
\begin{equation*}
 \int_{F}\tau_{2,h}\cdot\nu_{xy}p\,ds=\int_{F}(\tau_2-\widetilde{\tau}_{2,h})\cdot\nu_{xy}p\,ds\text{ for any }p\in Q_{1,0}(F),
      \end{equation*}
and for face $F$ that parallels the $(x,y)$-plane,
\begin{equation}\label{modifydegree}
\int_{F}\tau_{2,h}\cdot  p\,dxdy=\int_{F}(\tau_2-\widetilde{\tau}_{2,h})\cdot  p\,dxdy\text{ for any }p\in R_{xy}(\Delta_{xy}),
\end{equation}
\begin{equation*}
\int_{F}\tau_{3,h}p\,dxdy=\int_{F}(\tau_3-\widetilde{\tau}_{3,h})p\,dxdy\text{ for any }p\in P_1(x,y).
\end{equation*}
Note that \eqref{modifydegree} is a combination of (6) and a slight modification of (7) in Lemma \ref{degreeoffreedom}, replacing $p={\rm curl}b_{xy}\in\Psi_{k+1}(\Delta_{xy})$ with $p=(y,-x)^T$ there. This is valid because of the result \eqref{integralresult}. In addition, the remaining degrees of freedom vanish for $\tau_{1,h}, \tau_{2,h}$ and $\tau_{3,h}$.

Since $U_{k,h}\subset\Sigma_{k,h}$, we define $I_h\tau=\widetilde{I}_h\tau+\begin{pmatrix}
              \tau_{1,h} & \tau_{2,h} \\
              \tau_{2,h}^T & \tau_{3,h} \\
            \end{pmatrix}
$. An integration by parts immediately yields that \eqref{inteRig1} holds true.
The stability estimate follows from \eqref{SZbound} and the definition of the correction $\begin{pmatrix}
              \tau_{1,h} & \tau_{2,h} \\
              \tau_{2,h}^T & \tau_{3,h} \\
            \end{pmatrix}$.
\end{proof}
\begin{theorem}\label{Thm:inf} For $k\geq 1$, there exists a positive constant
  $C$ independent of the meshsize $h$ with
\begin{equation}
\sup\limits_{0\not=\tau_h\in\Sigma_{k, h}}\frac{(\d\tau_h,v_h)}
{\|\tau_h\|_{H(\d,\Om)}}\geq C\|v_h\|_{0,\Omega} \quad
 \text{ for any } v_h\in V_{k, h}.\nn
\end{equation}
\end{theorem}
\begin{proof} By the stability of the continuous formulation, see \cite{Arnold-Winther-conforming,Hu2015trianghigh}, there exists a $\tau\in H^1(\Omega;\mathbb{S})$ such  that
     \begin{equation*}
       \d\tau=v_h\text{ and }\|\tau\|_{1,\Omega}\leq C\|v_h\|_{0,\Omega}.
     \end{equation*}
  This plus Lemma \ref{step2} implies that
\begin{equation}
\int_K(\div I_h\tau-v_h)\cdot w dx=0 \text{ for any }  w\in RM(K) \text{ and any element } K
\end{equation}
and
\begin{equation}
\|I_h\tau\|_{H(\div, \Omega)}\leq C\|v_h\|_{0,\Omega}.
\end{equation}
By Lemma \ref{step1}, there exists a $\delta_h\in \Sigma_{k, h}$ such that
\begin{equation}
\div \delta_h=v_h-\d I_h\tau\text{ and } \|\delta_h\|_{H(\div, \Omega)}\leq C\|v_h-\d I_h\tau\|_{0, \Omega}.
\end{equation}
Let $\tau_h=I_h\tau+\delta_h$. Then we have $\d\tau_h=v_h$ and $\|\tau_h\|_{H(\div, \Omega)}\leq C\|v_h\|_{0, \Omega}$.
\end{proof}
\begin{remark}\label{Fortininte}
Similarly as mentioned in \cite{Hu2015trianghigh}, it follows from Lemma \ref{step1} and Lemma \ref{step2} that there exists an interpolation operator ${\rm\Pi}_h:H^1(\Omega;\mathbb{S})\rightarrow\Sigma_{k,h}$ such that
\begin{equation*}
  (\d(\tau-{\rm\Pi}_h\tau),v_h)_K=0\text{ for any }K\text{ and }v_h\in V_{k,h}
\end{equation*}
for any $\tau\in H^1(\Omega;\mathbb{S})$. Further, if $\tau\in H^{k+1}(\Omega;\mathbb{S})$, it holds that
\begin{equation*}
  \|\tau-{\rm\Pi}_h\tau\|_{0,\Omega}\leq Ch^{k+1}\|\tau\|_{k+1,\Omega}.
\end{equation*}
\end{remark}
\begin{theorem}\label{MainError} Let
  $(\sigma, u)\in\Sigma\times V$ be the exact solution of
   problem \meqref{eqn1} and $(\sigma_h, u_h)\in\Sigma_{k,h}\times
   V_{k,h}$ the finite element solution of \meqref{diseqn1}.  Then,
   for $k\ge 1$,
\an{ \lab{t1} \|\sigma-\sigma_h\|_{H({\rm div},\Omega)}
    + \|u-u_h\|_{0,\Omega}&\le     Ch^{k+1}(\|\sigma\|_{k+2,\Omega}+\|u\|_{k+1,\Omega}).
      }
\end{theorem}
\begin{proof}
We follow the standard error estimate of mixed finite element methods in \cite{Boffi2013}
\begin{equation*}
  \|\sigma-\sigma_h\|_{H({\rm div},\Omega)}
    + \|u-u_h\|_{0,\Omega}\leq C\inf_{\tau_h\in\Sigma_{k,h},v_h\in V_{k,h}}(\|\sigma-\tau_h\|_{H({\rm div},\Omega)}
    + \|u-v_h\|_{0,\Omega}).
\end{equation*}
Let $P_h$ denote the local $L^2$ projection operator, from $V$ to $V_{k,h}$, satisfying the error estimate
\begin{equation*}
  \|v-P_hv\|_{0,\Omega}\leq Ch^{k+1}\|v\|_{k+1,\Omega}\text{ for any }v\in H^{k+1}(\Omega;\R^3).
\end{equation*}
Choosing $\tau_h={\rm\Pi}_h\sigma$ where ${\rm\Pi}_h$ is defined in Remark \ref{Fortininte}, note that $\d{\rm\Pi}_h\sigma=P_h\d\sigma$, we have
\begin{equation*}
  \|\sigma-{\rm\Pi}_h\sigma\|_{H({\rm div},\Omega)}\leq Ch^{k+1}\|\sigma\|_{k+2,\Omega}.
\end{equation*}
Consequently, a  choice of $v_h=P_hu$ and $\tau_h={\rm\Pi}_h\sigma$ completes the proof.
\end{proof}

\section{Reduced mixed triangular prism elements}
\label{sec:reducedFEM}
In this section, we provide a family of reduced spaces of $\Sigma_{k,h}$ in \eqref{stressspace}. According to Lemma \ref{step1}, we know that we only need the degrees of freedom of bubble function space $\Sigma_{K,k,b}$ to deal with the space $RM^\perp(K)$. From the proof of Lemma \ref{step2}, we only need degrees of freedom on faces of the  lowest order element to deal with the rigid motion space $RM(K)$ on each element $K$. Hence the stress finite elements can be reduced by replacing $H_{k+2,h}$ and ${\rm BDM}_{k+1}$ in \eqref{stressspace} by $\widetilde{H}_{k+2,h}$ and $\widetilde{{\rm BDM}}_{k+1}$ as follows
\an{\nonumber
 \widetilde{H}_{k+2,h} :=\Big\{~\tau
    &\in H({\rm div}_{xy},\Omega_{xy};\mathbb{S}_{2}) \ \Big| \ \tau=\tau_\ell+\tau_b,
             \ \tau_\ell\in H_{k,h}, \\
     &
	\nonumber  \ \tau_b|_{\Delta_{xy}}\in H_{\Delta_{xy},k+2, b}  \text{ for any } \Delta_{xy}\in\mathcal{X}_h \Big\}, }
\an{\nonumber
 \widetilde{{\rm BDM}}_{k+1}: =\Big\{~\tau
    &\in H({\rm div}_{xy},\Omega_{xy};\mathbb{S}_{2}) \ \Big| \ \tau=\tau_\ell+\tau_b,
             \ \tau_\ell\in {\rm BDM}_{k}, \\
     &
	\nonumber  \ \tau_b|_{\Delta}\in {\rm BDM}_{\Delta_{xy},k+1, b}  \text{ for any } \Delta_{xy}\in\mathcal{X}_h \Big\}. }
\begin{remark}
We know that $H_{k,h}$ is defined for $k\geq 3$ in \eqref{Sh}. When $k=1,2$, we refer interested readers to \cite{HuZhang2015trianglow} for those two cases and omit the specific definitions herein. Thus, the degrees of freedom on each element of our lowest order case, which is of second order convergence, are 93 plus 33.
\end{remark}
We use $\widetilde{\Sigma}_{k,h}$ to denote the new stress spaces. The reduced elements preserve the same convergence order.
\begin{theorem}\label{MainErrorReduced} Let
  $(\sigma, u)\in\Sigma\times V$ be the exact solution of
   problem \meqref{eqn1} and $(\sigma_h, u_h)\in\widetilde{\Sigma}_{k,h}\times
   V_{k,h}$ the discrete solution by the reduced triangular prism elements.  Then,
   for $k\ge 1$,
\an{ \nonumber \|\sigma-\sigma_h\|_{H({\rm div,\Omega})}
    + \|u-u_h\|_{0,\Omega}&\le Ch^{k+1}(\|\sigma\|_{k+2,\Omega}+\|u\|_{k+1,\Omega}).
      }
\end{theorem}
\begin{remark}
For $k=1$, if we utilize the first order nonconforming stress space of  \cite{ArnoldWinther2003,Gopalakrishnan2011} in two dimensions instead of the first order conforming element of \cite{HuZhang2015trianglow}, we obtain a nonconforming mixed triangular prism element of first order convergence, with 81 plus 33 degrees of freedom. The following error estimate holds
\an{ \nonumber\|\sigma-\sigma_h\|_{0,\Omega}
    + \|u-u_h\|_{0,\Omega}&\le Ch\|u\|_{2,\Omega}.
      }
\end{remark}

\section{Nonconforming mixed tetrahedral elements}
\label{sec:tetrahedral element}
We construct a family of nonconforming mixed tetrahedral elements for the linear elasticity problem. The shape function spaces of the new stress spaces are different from those of the reduced elements
in \cite{Gopalakrishnan2011}.
\subsection{The discrete stress and displacement spaces}
For simplification, we still use some notations in previous sections and adapt them  to the current case. The domain $\Omega$ is subdivided by a family of shape regular tetrahedral meshes $\cT_h$. We introduce the displacement
space as the full $C^{-1}$-$P_k$ space
\begin{equation}\label{nondisplace}
  V_{k,h}^{\rm nc}:=\{v\in L^2(\Omega;\R^3)\ |\ v|_K\in P_{k}(K;\R^3)\text{ for all }K\in\cT_h\}.
\end{equation}
Let $\bx_i(1\leq i\leq 4)$ denote the vertices of tetrahedron $K$ and $\bft_{i,j}=\bx_j-\bx_i(i\neq j)$ the tangent vector of edge $\bx_i\bx_j$. Let $\lambda_i(1\leq i\leq4)$ denote the barycentric coordinates with respect to $\bx_i$.
Before defining the discrete stress space, we first propose the following shape function space on element $K$
\begin{equation*}
  \Sigma^{\rm nc}_k(K):=\{\tau=\sum_{1\leq i<j\leq 4}p_{ij}\bft_{i,j}\bft_{i,j}^T\ |\ p_{ij}\in P_{ij}\},
\end{equation*}
where
\begin{equation}\label{nonelementO}
\begin{split}
 P_{ij}:
 =&P_{k}(K;\R)+(\lambda_i-\lambda_j)\widetilde{P}_{k}(\lambda_\ell,\lambda_m)+\lambda_i\lambda_jP_{k-1}(K;\R)
 \end{split}
\end{equation}
and $\{\ell,m\}=\{1,2,3,4\}\backslash\{i,j\}$, see the definition of $\widetilde{P}_{k}(\lambda_\ell,\lambda_m)$ in \eqref{def:bary}.
We use the following degrees of freedom in \cite{Gopalakrishnan2011}:
\begin{itemize}
  \item[(1)] $\int_F\tau\nu\cdot v\,ds$ for any $v\in P_{k}(F;\R^3)$ and face $F\subset \partial K$,
  \item[(2)] $\int_K\tau:p\,dx$ for any $p\in H_{K,k+1,b}$,
  where
   \begin{equation*}
    H_{K,k+1,b}:=\sum_{1\leq i<j\leq 4}\lam_i\lam_jP_{k-1}(K;\R)\bft_{i,j}\bft_{i,j}^T.
    \end{equation*}
\end{itemize}
Note that the  conditions in (2) are slightly different from those of \cite{Gopalakrishnan2011}. The definition of $H_{K,k+1,b}$ follows from \cite{Hu2015trianghigh,HuZhang2015tre}, which satisfies
\begin{equation*}
H_{K,k+1,b}=\{\tau\in P_{k+1}(K;\mathbb{S})\ |\ \tau\nu|_{\partial K}=0\}\text{ and }
{\rm div}H_{K,k+1,b}=RM^\perp(K).
\end{equation*}
Here \begin{equation*}
RM^\perp(K):=\{v\in P_k(K;\R^3)\ |\ (v,w)_{K}=0\text{ for any }w\in RM(K)\}.
\end{equation*}
In order to count the dimensions of $\Sigma^{\rm nc}_k(K)$, we propose the following direct sum decomposition, which can be  easily checked.
\begin{lemma} It holds that
\begin{equation}\label{nonelement}
 P_{ij}
 =P_{k}(K;\R)\oplus(\lambda_i-\lambda_j)\widetilde{P}_{k}(\lambda_\ell,\lambda_m)\oplus \lambda_i\lambda_j P^\perp_{k-2}(K;\R),
\end{equation}
where $$P^\perp_{k-2}(K;\R):=\{p\in P_{k-1}(K;\R)\ |\ (p,q)_K=0\text{ for all }q\in P_{k-2}(K;\R)\}.$$
\end{lemma}
\begin{theorem}\label{Thm:non}
The conditions in (1) and (2) form unisolvent degrees of freedom for $\Sigma^{\rm nc}_k(K)$.
\end{theorem}
\begin{proof}
It follows immediately from the direct  sum decomposition \eqref{nonelement} that the dimensions of $\Sigma^{\rm nc}_k(K)$ are equal to the number of conditions in (1) and (2).  Next, we show that if the degrees of freedom all vanish for some $\tau\in\Sigma^{\rm nc}_k(K)$, then $\tau=0$.

Suppose  $\tau=\sum\limits_{1\leq i<j\leq 4}p_{ij}\bft_{i,j}\bft_{i,j}^T$ and $p_{ij}\in  P_{ij}$. In the following, we suppose $p_{ij}$ is defined for any $1\leq i\neq j\leq 4$ with $p_{ij}=p_{ji}$. Give vertex $\bx_i$, let $F_i$ denote the face opposite it. Note that for any
 $v\in P_{k}(F_i;\R^3)$
\begin{equation*}
0=\int_{F_i}\tau\nu\cdot v\,ds=\sum^4_{s=1,s\neq i}(\bft_{i,s}\cdot\nu)\int_{F_i}p_{is}(\bft_{i,s}\cdot v)\,ds.
\end{equation*}
Since $\bft_{i,s}(1\leq s\leq 4,s\neq i)$ are linearly independent, there exists a vector $\bfn$ satisfies $\bfn\cdot\bft_{i,j}=1$ and $\bfn\cdot\bft_{i,s}=0$ for $s=\ell,m$. By selecting $v=w_1\bfn$ in the above equation, we obtain
\begin{equation*}
 \int_{F_i}p_{ij}w_1\,ds=0\text{ for any }w_1\in  P_{k}(F_i;\R).
 \end{equation*}
 Similar arguments show that
\begin{equation*}
 \int_{F_j}p_{ij}w_2\,ds=0\text{ for any }w_2\in  P_{k}(F_j;\R).
 \end{equation*}
Thanks to the decomposition \eqref{nonelement}, $p_{ij}$ is of the form
 \begin{equation}\label{faceeq0}
 \begin{split}
 p_{ij}=\sum_{\alpha_\ell+\alpha_m+\alpha_i+\alpha_j=k}&C_{\alpha_\ell\alpha_m\alpha_i\alpha_j}
 \lambda_\ell^{\alpha_\ell}\lambda_m^{\alpha_m}\lambda_i^{\alpha_i}\lambda_j^{\alpha_j}\\
 &
 +(\lambda_i-\lambda_j)\sum_{\beta_\ell+\beta_m=k}D_{\beta_\ell\beta_m}
 \lambda_\ell^{\beta_\ell}\lambda_m^{\beta_m}+\lambda_i\lambda_jp_1,
 \end{split}
\end{equation}
where $p_1\in P_{k-2}^\perp(K;\R)$. Here we use the fact that
\begin{equation*}
  P_{k}(K;\R)={\rm span}\left\{\prod^4_{s=1}\lambda_s^{\alpha_s},\sum^4_{s=1}\alpha_s=k\right\}.
\end{equation*}
The expression in \eqref{faceeq0} immediately yields that for any $w_1\in  P_{k}(F_i;\R)$
\begin{equation}
\begin{split}
\label{faceeq1}
  \int_{F_i}\big(\sum_{\alpha_\ell+\alpha_m+\alpha_j=k}&C_{\alpha_\ell\alpha_m0\alpha_j}
  \lambda_\ell^{\alpha_\ell}\lambda_m^{\alpha_m}\lambda_j^{\alpha_j}\\
  &-\lambda_j\sum_{\beta_\ell+\beta_m=k}D_{\beta_\ell\beta_m}
 \lambda_\ell^{\beta_\ell}\lambda_m^{\beta_m}\big)w_1\,ds=0,\\
 \end{split}
\end{equation}
and for any $w_2\in  P_{k}(F_j;\R)$
\begin{equation}\begin{split}
 \label{faceeq2}
  \int_{F_j}\big(\sum_{\alpha_\ell+\alpha_m+\alpha_i=k}&C_{\alpha_\ell\alpha_m\alpha_i0}
  \lambda_\ell^{\alpha_\ell}\lambda_m^{\alpha_m}\lambda_i^{\alpha_i}\\
  &+\lambda_i\sum_{\beta_\ell+\beta_m=k}D_{\beta_\ell\beta_m}
 \lambda_\ell^{\beta_\ell}\lambda_m^{\beta_m}\big)w_2\,ds=0.\end{split}
\end{equation}
Note that the restrictions of $\lambda_s(s\neq i)$ on $F_i$ and $\lambda_s(s\neq j)$ on $F_j$ are exactly the barycentric coordinates on $F_i$ and $F_j$, respectively. Then, it holds that
\begin{equation*}
  P_{k}(F_j;\R)={\rm span}\left\{\prod^4_{s=1,s\neq j}(\lambda_s|_{F_j})^{\alpha_s},\sum^4_{s=1,s\neq j}\alpha_s=k\right\}.
\end{equation*}
Therefore, replacing $\lambda_i$  by $\lambda_j$ and the domain $F_j$ by $F_i$ in \eqref{faceeq2}, we have that for any $w_1\in P_{k}(F_i;\R)$
\begin{equation*}
 \int_{F_i}\big(\sum_{\alpha_\ell+\alpha_m+\alpha_j=k}C_{\alpha_\ell\alpha_m\alpha_j0}
  \lambda_\ell^{\alpha_\ell}\lambda_m^{\alpha_m}\lambda_j^{\alpha_j}+\lambda_j\sum_{\beta_\ell+\beta_m=k}D_{\beta_\ell\beta_m}
 \lambda_\ell^{\beta_\ell}\lambda_m^{\beta_m}\big)w_1\,ds=0.
\end{equation*}
By comparing  with \eqref{faceeq1}, we get that
\begin{equation}\label{keypoint}
  \int_{F_i}\sum_{\alpha_\ell+\alpha_m+\alpha_j=k}(C_{\alpha_\ell\alpha_m0\alpha_j}+C_{\alpha_\ell\alpha_m\alpha_j0})\lambda_\ell^{\alpha_\ell}\lambda_m^{\alpha_m}\lambda_j^{\alpha_j}w_1\,ds=0.
\end{equation}
Since $w_1\in P_k(F_i;\R)$ is arbitrary, this yields that
\begin{equation}\label{coff00}
C_{\alpha_\ell\alpha_m00}=0\text{ for any }\alpha_\ell+\alpha_m=k
\end{equation}
and
$$C_{\alpha_\ell\alpha_m0\alpha_j}=-C_{\alpha_\ell\alpha_m\alpha_j0}\text{ for any }\alpha_\ell+\alpha_m+\alpha_j=k\text{ with }\alpha_j\geq1.$$
Inserting \eqref{coff00} into \eqref{faceeq1} shows that
\begin{equation*}
   \int_{F_i}\lambda_j\bigg(\sum_{\substack{\alpha_\ell+\alpha_m+\alpha_j=k\\ \alpha_j\geq 1}}C_{\alpha_\ell\alpha_m0\alpha_j}\lambda_\ell^{\alpha_\ell}\lambda_m^{\alpha_m}\lambda_j^{\alpha_j-1}
   -\sum_{\beta_\ell+\beta_m=k}D_{\beta_\ell\beta_m}
 \lambda_\ell^{\beta_\ell}\lambda_m^{\beta_m}\bigg)w_1\,ds=0,\\
\end{equation*}
which implies
\begin{equation*}
 C_{\alpha_\ell\alpha_m0\alpha_j}=C_{\alpha_\ell\alpha_m\alpha_j0}=D_{\beta_\ell\beta_m}=0.
\end{equation*}
Consequently, we have
  \begin{equation*}
 p_{ij}=\lambda_i\lambda_j\sum_{\substack{\alpha_\ell+\alpha_m+\alpha_i+\alpha_j=k\\
 \alpha_i,\alpha_j\geq 1}}C_{\alpha_\ell\alpha_m\alpha_i\alpha_j}
 \lambda_\ell^{\alpha_\ell}\lambda_m^{\alpha_m}\lambda_i^{\alpha_i-1}\lambda_j^{\alpha_j-1}
 +\lambda_i\lambda_jp_1.
\end{equation*}
  Now, using the conditions in (2) and $p_1\in P_{k-1}(K;\R)$, we finally find that $p_{ij}=0$, which completes the proof.
\end{proof}
We define the nonconforming stress space for $k\geq 1$
\an{\lab{nonSh}
   \Sigma^{\rm nc}_{k,h}:=\Big\{~\tau
    &\in L^2(\Omega;\mathbb{S}) \ \Big| \ \tau|_K\in \Sigma^{\rm nc}_{k}(K)\text{ for any }K\in\cT_h, \text{ and the moments}\\
     & \text{of $\tau\nu$ up to degree $k$ are continuous across internal faces}\Big\}.
	\nonumber   }
\begin{remark}The shape function space $\Sigma_k^{\rm nc}(K)$ can be alternatively  defined, replacing $(\lambda_i-\lambda_j)\widetilde{P}_{k}(\lambda_\ell,\lambda_m)$  in \eqref{nonelementO} by $(c_1\lambda_i+c_2\lambda_j)\widetilde{P}_{k}(\lambda_\ell,\lambda_m)$ with any two different constants $c_1$ and $c_2$. This is valid since we can still get the key result \eqref{keypoint} in the proof of Theorem \ref{Thm:non}. When $c_1=0$ or $c_2=0$, we recover the reduced elements in \cite{Gopalakrishnan2011}.
\end{remark}
Our method finds $(\sigma_h,u_h)\in \Sigma^{\rm nc}_{k,h}\times V^{\rm nc}_{k,h}$ such that
\an{\left\{ \ad{
  &(A\sigma_h,\tau)+({\rm div}_h\tau,u_h)= 0 && \hbox{for all \ } \tau\in \Sigma^{\rm nc}_{k,h},\\
   &({\rm div}_h\sigma_h,v)= (f,v) && \hbox{for all \ } v\in V^{\rm nc}_{k,h}. }
   \right.\lab{diseqnNon}
}
Here ${\rm div}_h$ denotes the discrete counterpart of the divergence operator ${\rm div}$. We have the following error estimate as in \cite{Gopalakrishnan2011}.
\begin{theorem}\label{thm:non}
Suppose $k\geq1$ and $\sigma\in H^1(\Omega;\mathbb{S})$. Then for any $1\leq r\leq k$,
\begin{equation*}
  \|\sigma-\sigma_h\|_{0,\Omega}+\|u-u_h\|_{0,\Omega}\leq Ch^r\|u\|_{r+1,\Omega}.
\end{equation*}
Moreover, if full elliptic regularity \cite[(3.19)]{Gopalakrishnan2011} holds, then
\begin{equation*}
 \|u-u_h\|_{0,\Omega}\leq Ch^{r+1}\|u\|_{r+1,\Omega}.
\end{equation*}
\end{theorem}
\subsection{Basis functions of the stress space for $k=1$}The degrees of freedom of $\Sigma^{\rm nc}_1(K)$ are 42. The six interior bubble functions on tetrahedron $K$ are as follows
\begin{equation*}
\lambda_i\lambda_j\bft_{i,j}\bft_{i,j}^T,\ 1\leq i<j\leq 4.
\end{equation*}
Below we give the basis functions associated with face $F_1$ as an example. The three vertices of $F_1$ are $\bx_2,\bx_3,\bx_4$ and $h_1$ denotes the height from $\bx_1$ to $F_1$. By direct computation, we obtain nine functions $\phi_{1,j}^{(s)}\in P_{1j}$ for $2\leq j\leq 4, 1\leq s\leq3$ that
\begin{eqnarray*}
  \phi_{1,2}^{(1)}&=&27\lambda_2-9\lambda_1-3(\lambda_3+\lambda_4)+30(\lambda_1-\lambda_2)(\lambda_3+\lambda_4),\\
  \phi_{1,2}^{(2)}&=&-15\lambda_2+9\lambda_1+9\lambda_3-3\lambda_4-60(\lambda_1-\lambda_2)\lambda_3,\\
  \phi_{1,2}^{(3)}&=&-15\lambda_2+9\lambda_1-3\lambda_3+9\lambda_4-60(\lambda_1-\lambda_2)\lambda_4,\\
  \phi_{1,3}^{(1)}&=&-15\lambda_3+9\lambda_1-3\lambda_4+9\lambda_2-60(\lambda_1-\lambda_3)\lambda_2,\\
  \phi_{1,3}^{(2)}&=&27\lambda_3-9\lambda_1-3(\lambda_4+\lambda_2)+30(\lambda_1-\lambda_3)(\lambda_4+\lambda_2),\\
  \phi_{1,3}^{(3)}&=&-15\lambda_3+9\lambda_1+9\lambda_4-3\lambda_2-60(\lambda_1-\lambda_3)\lambda_4,\\
  \phi_{1,4}^{(1)}&=&-15\lambda_4+9\lambda_1+9\lambda_2-3\lambda_3-60(\lambda_1-\lambda_4)\lambda_2,\\
  \phi_{1,4}^{(2)}&=&-15\lambda_4+9\lambda_1-3\lambda_2+9\lambda_3-60(\lambda_1-\lambda_4)\lambda_3,\\
  \phi_{1,4}^{(3)}&=&27\lambda_4-9\lambda_1-3(\lambda_2+\lambda_3)+30(\lambda_1-\lambda_4)(\lambda_2+\lambda_3),
\end{eqnarray*}
which satisfy
\begin{equation}\label{basP1}
\int_{F_j}\phi_{1,j}^{(s)}w\,ds=0\text{ for any }w\in P_1(F_j;\R)
\end{equation}
and for any $1\leq r\leq 3$
\begin{equation*}
\frac{1}{2|F_1|}\int_{F_1}\phi_{1,j}^{(s)}\lambda_{r+1}\,ds=\delta_{s,r}:=\begin{cases}
1&\text{ if }s=r,\\
0&\text{ otherwise}.
\end{cases}
\end{equation*}
Define $\Phi_{1,j}^{(s)}=\phi_{1,j}^{(s)}\bft_{1,j}\bft_{1,j}^T$. It follows from \eqref{basP1} that
\begin{equation*}
  \int_{F_r}\Phi_{1,j}^{(s)}\nu\cdot v\,ds=0\text{ for any }v\in P_1(F_r;\R^3)\text{ and }2\leq r\leq 4.
\end{equation*}

We want to establish the global basis across element interfaces. To this end, we find three arbitrary linear independent vectors $\bfn_1,\bfn_2$ and $\bfn_3$ associated with face $F_1$. Let
$\{\Phi_\ell\}_{\ell=1}^9$ denote the collection of the nine matrix-valued functions $\Phi_{1,j}^{(s)}$ and
$$\{q_\ell\}_{\ell=1}^9=\{\lambda_2\bfn_1,\lambda_3\bfn_1,\lambda_4\bfn_1,\lambda_2\bfn_2,\lambda_3\bfn_2,
\lambda_4\bfn_2,\lambda_2\bfn_3,\lambda_3\bfn_3,\lambda_4\bfn_3\}.$$
Define the matrix $H=(h_{\ell m})\in\R^{9\times9}$, where
\begin{equation*}
  h_{\ell m}=\int_{F_1}q_\ell\cdot(\Phi_m\nu)\,ds,\ 1\leq \ell,m\leq 9.
\end{equation*}
We write down the explicit expression of $H$. For example $h_{11}=\int_{F_1}\lambda_2\bfn_1\cdot(\Phi_{1,2}^{(1)}\nu)\,ds=2h_1|F_1|\bft_{1,2}\cdot\bfn_1=6|K|\bft_{1,2}\cdot\bfn_1$, then
\begin{equation*}
 H=6|K|\begin{pmatrix}
     (\bft_{1,2}\cdot\bfn_1) \delta& (\bft_{1,3}\cdot\bfn_1) \delta  & (\bft_{1,4}\cdot\bfn_1)\delta\\
     (\bft_{1,2}\cdot\bfn_2) \delta& (\bft_{1,3}\cdot\bfn_2)\delta & (\bft_{1,4}\cdot\bfn_2) \delta\\
     (\bft_{1,2}\cdot\bfn_3)\delta& (\bft_{1,3}\cdot\bfn_3) \delta & (\bft_{1,4}\cdot\bfn_3) \delta \\
   \end{pmatrix}.
\end{equation*}
Here $\delta\in\R^{3\times3}$ denotes the identity matrix. Note that $\bft_{1,j}(2\leq j\leq 4)$ are linearly independent, we have
\begin{equation}\label{mainv}
 S=(s_{ij})=\begin{pmatrix}
     \bft_{1,2}\cdot\bfn_1 & \bft_{1,3}\cdot\bfn_1   & \bft_{1,4}\cdot\bfn_1  \\
     \bft_{1,2}\cdot\bfn_2& \bft_{1,3}\cdot\bfn_2 & \bft_{1,4}\cdot\bfn_2 \\
     \bft_{1,2}\cdot\bfn_3& \bft_{1,3}\cdot\bfn_3  & \bft_{1,4}\cdot\bfn_3\\
   \end{pmatrix}^{-1},
\end{equation}
which means that we transform $\bft_{1,2},\bft_{1,3},\bft_{1,4}$ to $\bfn_1,\bfn_2,\bfn_3$. Further, it can be checked that the inverse of $H$ is as follows
\begin{equation*}
 H^{-1}=(b_{\ell m})=\frac{1}{6|K|}\begin{pmatrix}
     s_{11} \delta&  s_{12} \delta  &  s_{13} \delta\\
      s_{21} \delta&  s_{22} \delta  &  s_{23}\delta \\
      s_{31} \delta&  s_{32} \delta&  s_{33} \delta \\
   \end{pmatrix}.
\end{equation*}
 Consequently, the desired nine basis functions associated with $F_1$ are
 \begin{equation}\label{basex}
 \Psi_m=\sum^{9}_{\ell=1}b_{\ell m}\Phi_\ell,\ 1\leq m\leq 9,
\end{equation}
satisfying
 \begin{equation*}
\int_{F_1}q_\ell\cdot(\Psi_m\nu)\,ds=\delta_{\ell,m},\ 1\leq \ell\leq 9.
\end{equation*}

For any two elements $K$ and $K'$ which share the common face $F_1$, vectors $\bfn_i(1\leq i\leq 3)$ associated with $F_1$ can be chosen as the dual basis of tangent vectors $\bft_{1,2}$, $\bft_{1,3}$ and  $\bft_{1,4}$ of $K$. In this case, on element $K$, we have $S=\delta$ and thus
\begin{equation*}
H^{-1}=\frac{1}{6|K|}\begin{pmatrix}
    \delta & 0   & 0\\
     0& \delta & 0 \\
     0& 0  & \delta\\
   \end{pmatrix}.
\end{equation*}
Hence the basis functions $\Psi_m$ associated with $F_1$ on element $K$ are exactly $\frac{1}{6|K|}\Phi_m$.
In conclusion, to get the global basis functions, we first need to find nine functions $\phi_{i,j}^{(s)}(j\neq i)$ for face $F_i$ as above, which are defined in barycentric coordinates and independent of all elements. Then, we need to compute the inverse of the $3\times 3$ matrix  as in \eqref{mainv} for face $F_i$ on each element. The basis functions follow from \eqref{basex}.
\subsection{The nonconforming mixed triangular elements}Following similar arguments, we construct the nonconforming stress space on triangular meshes. In this subsection, we suppose  domain $\Omega\subset\R^2$ is subdivided by a family of shape regular triangular meshes $\cT_h$. Notations here are similar to those in the three dimensional case. Given triangle $K\in\cT_h$,  we define
the shape function space of order $k\geq 1$
\begin{equation*}
  \Sigma^{\rm nc}_k(K):=\{\tau=\sum_{1\leq i<j\leq 3}p_{ij}\bft_{i,j}\bft_{i,j}^T\ |\ p_{ij}\in P_{ij}\},
\end{equation*}
where
\begin{equation*}
\begin{split}
 P_{ij}:
 =&P_{k}(K;\R)+{\rm span}\{(\lambda_i-\lambda_j)\lambda_\ell^k\}+\lambda_i\lambda_jP_{k-1}(K;\R)
 \end{split}
\end{equation*}
and $\{\ell\}=\{1,2,3\}\backslash\{i,j\}$.
The degrees of freedom are as follows:
\begin{itemize}
  \item[(1)] $\int_e\tau\nu\cdot v\,ds$ for any $v\in P_{k}(e;\R^2)$ and edge $e\subset \partial K$,
  \item[(2)] $\int_{K}\tau:p\,dx$ for any $p\in H_{K,k+1,b}$,
  where
   \begin{equation*}
    H_{K,k+1,b}=\sum_{1\leq i<j\leq 3}\lam_i\lam_jP_{k-1}(K;\R)\bft_{i,j}\bft_{i,j}^T.
    \end{equation*}
\end{itemize}
The proof of  unisolvence is similar as in Theorem \ref{Thm:non}. Then we define the stress space for $k\geq 1$
\an{\nonumber
   \Sigma^{\rm nc}_{k,h}:=\Big\{&~\tau
    \in L^2(\Omega;\mathbb{S}_2) \ \Big| \ \tau|_{K}\in \Sigma^{\rm nc}_{k}(K)\text{ for any }K\in\cT_h, \text{ and the moments}\\
     & \text{ of $\tau\nu$ up to degree $k$ are continuous across internal edges}\Big\}.
	\nonumber   }

Below we give the basis functions of $\Sigma^{\rm nc}_1(K)$. The three interior bubble functions are
\begin{equation*}
\lambda_i\lambda_j\bft_{i,j}\bft_{i,j}^T,\ 1\leq i<j\leq 3.
\end{equation*}
 We give the basis functions associated with edge $e_1$ as an example. The two vertices of $e_1$ are $\bx_2,\bx_3$. By direct computation, we obtain four functions $\phi_{1,j}^{(s)}\in P_{1j}$ for $2\leq j\leq 3, 1\leq s\leq2$ that
\begin{eqnarray*}
  \phi_{1,2}^{(1)}&=&5\lambda_2-\lambda_1-\lambda_3+6(\lambda_1-\lambda_2)\lambda_3,\\
  \phi_{1,2}^{(2)}&=&-4\lambda_2+2\lambda_1+2\lambda_3-12(\lambda_1-\lambda_2)\lambda_3,\\
  \phi_{1,3}^{(1)}&=&-4\lambda_3+2\lambda_1+2\lambda_2-12(\lambda_1-\lambda_3)\lambda_2,\\
  \phi_{1,3}^{(2)}&=&5\lambda_3-\lambda_1-\lambda_2+6(\lambda_1-\lambda_3)\lambda_2,
  \end{eqnarray*}
  which satisfy
\begin{equation*}
\int_{e_j}\phi_{1,j}^{(s)}w\,ds=0\text{ for any }w\in P_1(e_j;\R)
\end{equation*}
and for any $1\leq r\leq 2$
\begin{equation*}
\frac{1}{|e_1|}\int_{e_1}\phi_{1,j}^{(s)}\lambda_{r+1}\,ds=\delta_{s,r}.
\end{equation*}
Here $e_j$ denotes the edge of $K$ opposite vertex $x_j$. Define $\Phi_{1,j}^{(s)}=\phi_{1,j}^{(s)}\bft_{1,j}\bft_{1,j}^T$.
We find two arbitrary linear independent vectors $\bfn_1$ and $\bfn_2$ associated with edge $e_1$. Let
$\{\Phi_\ell\}_{\ell=1}^4$ denote the collection of the four matrix-valued functions $\Phi_{1,j}^{(s)}$ and
$\{q_\ell\}_{\ell=1}^4=\{\lambda_2\bfn_1,\lambda_3\bfn_1,\lambda_2\bfn_2,\lambda_3\bfn_2\}$.
Define the matrix $H=(h_{\ell m})\in\R^{4\times4}$, where
\begin{equation*}
  h_{\ell m}=\int_{F_1}q_\ell\cdot(\Phi_m\nu)\,ds,\ 1\leq \ell,m\leq 4.
\end{equation*}
We write down the explicit expression of $H$ that
\begin{equation*}
 H=2|K|\begin{pmatrix}
     (\bft_{1,2}\cdot\bfn_1) \delta& (\bft_{1,3}\cdot\bfn_1)\delta \\
     (\bft_{1,2}\cdot\bfn_2) \delta& (\bft_{1,3}\cdot\bfn_2)\delta \\
   \end{pmatrix}.
\end{equation*}
Here $\delta\in\R^{2\times2}$ denotes the identity matrix. Note that $\bft_{1,j}(2\leq j\leq 3)$ are linearly independent, we have
\begin{equation*}
 S=(s_{ij})=\begin{pmatrix}
     \bft_{1,2}\cdot\bfn_1 & \bft_{1,3}\cdot\bfn_1    \\
     \bft_{1,2}\cdot\bfn_2& \bft_{1,3}\cdot\bfn_2 \\
   \end{pmatrix}^{-1},
\end{equation*}
which means that we transform $\bft_{1,2},\bft_{1,3}$ to $\bfn_1,\bfn_2$. Further,  the inverse of $H$ is as follows
\begin{equation*}
 H^{-1}=(b_{\ell m})=\frac{1}{2|K|}\begin{pmatrix}
     s_{11} \delta&  s_{12} \delta  \\
      s_{21} \delta&  s_{22} \delta \\
   \end{pmatrix}.
\end{equation*}
 Consequently, the desired four basis functions associated with $e_1$ are
 \begin{equation*}
 \Psi_m=\sum^{4}_{\ell=1}b_{\ell m}\Phi_\ell,\ 1\leq m\leq 4,
\end{equation*}
satisfying
 \begin{equation*}
\int_{e_1}q_\ell\cdot(\Psi_m\nu)\,ds=\delta_{\ell,m},\ 1\leq \ell\leq 4.
\end{equation*}
As mentioned in the three dimensional case, $\bfn_1$ and $\bfn_2$ can be chosen as the dual basis of tangent vectors $\bft_{1,2}$ and $\bft_{1,3}$ of some element that share $e_1$.
\section{Numerical results}
\label{sec:numerical}
We compute one example in 3D, by the lowest order case of the two family elements, respectively.
It is a pure displacement problem  on the unit cube
   $\Omega=(0,1)^3$ with a homogeneous boundary condition
         that $u\equiv 0$ on $\partial\Omega$.
In the computation, let
   \a{
      A \sigma &= \frac 1{2\mu} \left(
       \sigma - \frac{\lambda}{2\mu + 3\lambda} \operatorname{tr}(\sigma)
        \delta \right),  }
  where $\delta=\p{1 &0&0\\0&1&0\\0&0&1}$, and $\mu=1/2$ and $\lambda=1$ are the
    Lam\'e constants.

Let  the  exact solution on the unit square $[0,1]^3$ be
   \e{\lab{e1}   u= \p{2^4 \\2^5 \\ 2^6 } x(1-x)y(1-y)z(1-z). }
Then, the true stress function $\sigma$
     and the load function $f$ are defined by the equations in
    \meqref{eqn1},   for the given  solution $u$.

 \subsection{The mixed triangular prism element}We use the mixed triangular prism element of $k=1$ in Section \ref{sec:finite element}. In the computation, each mesh is refined into a half-sized mesh uniformly, see the initial mesh in Figure \ref{grid}.
In Table \mref{b1}, the errors and the convergence order
   in various norms are listed  for the true solution \meqref{e1},
   by the mixed finite element in \meqref{stressspace} and \meqref{Vh}, with
    $k=1$ there.
The optimal order of convergence is achieved in Table \mref{b1},
   coinciding with   Theorem \mref{MainError}.
   \subsection{The nonconforming mixed tetrahedral element}We compute the example on tetrahedral meshes by the lowest order nonconforming mixed element. The computational results are listed in Table \ref{b2}, which verifies Theorem \ref{thm:non}.
   \begin{table}[!ht]
  \caption{ The error and the order of convergence by the mixed triangular prism element,
    $k=1$ in \meqref{stressspace} and \eqref{Vh}, for \meqref{e1}.}
\lab{b1}
\begin{center}  \begin{tabular}{c|cc|cc|cc}  
\hline &  $ \|\sigma -\sigma_h\|_{0,\Omega}$ & $h^n$  & $ \|u- u_h\|_{0,\Omega}$ &$h^n$ &
    $ \|\d(\sigma -\sigma_h)\|_{0,\Omega}$ & $h^n$  \\ \hline
 1&    1.61682569&0.0&    0.21093411&0.0&    6.10467990&0.0\\
 2&     0.48388087&1.74&    0.06461602&1.71&    1.74304423&1.81\\
 3&     0.12795918&1.92&    0.01699145&1.92&    0.45537323&1.94\\
 4&     0.03244990&1.98&     0.00429655&1.98&     0.11514501&1.98\\
 5&     0.00814520&1.99&     0.00429654&1.99&     0.02886873&1.99\\\hline
\end{tabular}\end{center} \end{table}

   \begin{table}[!ht]
  \caption{ The error and the order of convergence by the nonconforming mixed tetrahedral element,
    $k=1$ in \eqref{nondisplace} and  \meqref{nonelement}, for \meqref{e1}.}
\lab{b2}
\begin{center}  \begin{tabular}{c|cc|cc|cc}  
\hline &  $ \|\sigma -\sigma_h\|_{0,\Omega}$ & $h^n$  & $ \|u- u_h\|_{0,\Omega}$ &$h^n$ &
    $ \|\d(\sigma -\sigma_h)\|_{0,\Omega}$ & $h^n$  \\ \hline
 1&    1.56676383&0.0&    0.28991289&0.0&    9.20147348&0.0\\
 2&     0.78169912&1.00&    0.09157813&1.66&    2.89615493&1.67\\
 3&     0.34907155&1.16&    0.02569030&1.83&    0.77454646&1.90\\
 4&     0.16459839&1.08&    0.00660946&1.96&    0.19693780&1.97\\
 5&     0.08060083&1.03&    0.00166329&1.99&    0.04944284&1.99\\ \hline
 \end{tabular}\end{center} \end{table}
 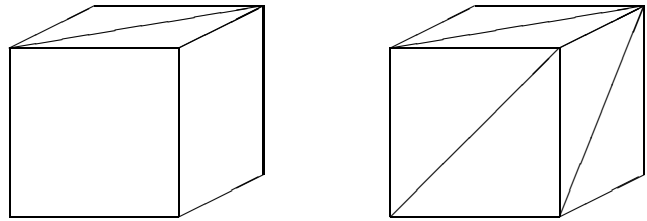
\begin{figure}[htb] \setlength\unitlength{0.8pt}
  \begin{center}  \begin{picture}(320,100)(0,0)
 \put(0,0){\begin{picture}(100,100)(0,0)
     \multiput(0,0)(80,0){2}{\line(0,1){80}}     \multiput(0,0)(0,80){2}{\line(1,0){80}}
    \multiput(80,0)(0,80){2}{\line(2,1){40}}    \multiput(0,80)(80,0){2}{\line(2,1){40}}
    \multiput(80,0)(40,20){2}{\line(0,1){80}}    \multiput(0,80)(40,20){2}{\line(1,0){80}}
      \multiput(0,80)(80,0){1}{\line(6,1){120}}
   \end{picture} }
 \put(180,0){\begin{picture}(100,100)(0,0)
     \multiput(0,0)(80,0){2}{\line(0,1){80}}     \multiput(0,0)(0,80){2}{\line(1,0){80}}
    \multiput(80,0)(0,80){2}{\line(2,1){40}}    \multiput(0,80)(80,0){2}{\line(2,1){40}}
    \multiput(80,0)(40,20){2}{\line(0,1){80}}    \multiput(0,80)(40,20){2}{\line(1,0){80}}
    \multiput(0,0)(40,0){1}{\line(1,1){80}}
    \multiput(80,0)(0,80){1}{\line(2,5){40}}    \multiput(0,80)(80,0){1}{\line(6,1){120}}
   \end{picture} }
    \end{picture} \end{center}
\caption{ \lab{grid} The initial meshes for
 the triangular prism and tetrahedral partitions, respectively.}
\end{figure}

\end{document}